\documentclass[preprint,12pt]{elsarticleq}

\usepackage{amssymb}
 \usepackage{amsthm}
\newtheorem{prop}{Proposition}[section]
\newtheorem{lemma}[prop]{Lemma}
\newtheorem{theorem}[prop]{Theorem}

\newdefinition{defi}[prop]{Definition}
\newdefinition{remark}[prop]{Remark}
\newdefinition{example}[prop]{Example}

\newcommand{\R}{\mathbb{R}}       %reals
\newcommand{\HH}{\mathbb{H}}    %quaternions
\newcommand{\cat}{\mathop{\mathrm{cat}}}     %category
\newcommand{\id}{\mathrm{id}} %identity
\newcommand{\SP}{\mathop{\mathit{Sp}}} %symplectic group
\newcommand{\rank}{\mathop{\mathrm{rank}}} 
  
\newcommand{\module}[1]{\vert #1 \vert}

\begin{document}

\begin{frontmatter}


% use the tnoteref command within \title for footnotes;
% use the tnotetext command for theassociated footnote;
% use the fnref command within \author or \address for footnotes;
% use the fntext command for theassociated footnote;

% use the corref command within \author for corresponding author footnotes;
% use the cortext command for theassociated footnote;
% and the form \ead[url] for the home page

\title{Symplectic matrices with predetermined left eigenvalues\tnoteref{label1}}
\tnotetext[label1]{Partially supported by FEDER and MICINN Spain Research Project MTM2008-05861}
% \author{Name\corref{cor1}\fnref{label2}}
\author{E. Mac\'{\i}as-Virg\'os}
\ead{quique.macias@usc.es}
 \ead[url]{web.usc.es/~xtquique}
% \address{Address\fnref{label3}}

\author{M. J. Pereira-S\'aez}
\ead{mariajose.pereira@usc.es}

%\title{}


%\author{}

%\address{}
\address{
Institute of Mathematics.
Department of Geometry and Topology.\\
University of Santiago de Compostela.
15782- SPAIN}

\begin{abstract}
% Text of abstract
We prove that given four arbitrary quaternion numbers of norm $1$ there always exists a $2\times 2$ symplectic matrix for which those numbers are left eigenvalues.  The proof is constructive. An application to the LS category of Lie groups is given.
\end{abstract} 

\begin{keyword}
quaternion\sep left eigenvalue\sep symplectic group\sep LS category

% MSC codes here, in the form: \MSC code \sep code
% or \MSC[2008] code \sep code (2000 is the default)
\MSC 15A33\sep 11R52 \sep 55M30 
\end{keyword}

\end{frontmatter}

% \linenumbers

% main text

%Article structure

%Introduction
%State the objectives of the work and provide an adequate background, avoiding a detailed literature survey or a summary of the results.

%Results
%Results should be clear and concise.

%Conclusions
%The main conclusions of the study may be presented in a short Conclusions section, which may stand alone or form a subsection of a Discussion or Results and Discussion section.

%\section{}
%\label{}

\section{Introduction}

Left eigenvalues of quaternionic matrices are only partially understood.
While their existence is guaranteed by a result from Wood \cite{WOOD}, many usual properties of right eigenvalues are no longer valid in this context, see Zhang's paper \cite{ZHANG} for a detailed account. In particular, a  matrix may have infinite left eigenvalues (belonging to different similarity classes), as has been proved by Huang and So \cite{HUSO}. By using this result, the authors characterized in   \cite{MACPER}  the symplectic $2\times 2$ matrices which have an infinite spectrum. In the present paper we prove that given four arbitrary quaternions of norm $1$ there always exists a matrix in $\SP(2)$ for which those quaternions are left eigenvalues.  The proof is constructive. This non-trivial result is of interest for the computation of the so-called LS category, as we  explain in the last section of the paper.

\section{Left eigenvalues of symplectic matrices}
Let $Sp(2)$ be the Lie group of $2\times 2$ symplectic matrices, that is quaternionic matrices such that $AA^*=I$, where $A^*=\bar A^t$ is the conjugate transpose.

By definition, a quaternion $\sigma$ is a left eigenvalue of the matrix $A$ if there exists a vector $v\in\HH^2$, $v\neq 0$, such that $Av=\sigma v$; equivalently, $A-\sigma I$ is not invertible.

It is easy to prove that the left eigenvalues of a symplectic matrix must have norm $1$.

\begin{remark} $\HH^2$ will be always considered as a {\em right} quaternionic vector space, endowed with the product $\langle v,w\rangle=v^*w$.
\end{remark}

The following Theorem was proven by the authors in \cite{MACPER}, using the results of Huang and So in \cite{HUSO}.

\begin{theorem}\label{INFINITE}
 The only $2\times 2$ symplectic matrices with an infinite number of left eigenvalues are those of the form
$$
L_q\circ R_\theta=\left[\matrix{ q\cos\theta & -q\sin\theta\cr
q\sin\theta &  q\cos\theta} \right],\quad   q\in\HH, \module{q}=1, \quad \theta\in\R, \sin\theta\neq 0.
$$
Any other symplectic matrix has one or two left eigenvalues.
\end{theorem}

The matrix $L_q\circ R_\theta$ above corresponds to the composition of a real rotation $R_\theta\neq\pm \mathrm{id}$ with  a left translation $L_q$, $\module{q}=1$. We need to characterize its eigenvalues.

%Its left eigenvalues are the quaternions
%$$\sigma = q(\cos\theta + \sin\theta\omega)$$
%where $\omega$ is a quaternion with real part $\Re(\omega)=0$ and norm $\module{\omega}=1$. 

\begin{lemma}\label{REAL} Let  the matrix $A=L_q\circ R_\theta$ be as in Theorem \ref{INFINITE} and let  $\sigma\in\HH$ be a quaternion. The following conditions are equivalent:
\begin{enumerate}
\item
$\sigma$  is a left eigenvalue of $A$;
 \item
$\sigma = q(\cos\theta +\sin\theta\cdot\omega)$ with $\omega\in\langle i,j,k\rangle_\R$, $\vert\omega\vert=1$;
\item
$\vert\sigma\vert=1$ and 
$\Re(\bar{q}\sigma)=\cos\theta$;
\item
$\bar q\sigma$ is conjugate to $\cos\theta + i\,\sin\theta$.
\end{enumerate}
\end{lemma}

\begin{proof} 
Part 2 can be checked by a direct computation from the definition or by using the results  in \cite{HUSO};
part 3 follows because $\bar{q}\sigma-\cos\theta$ has not real part;
finally part 4 is proved from the fact that two quaternions are conjugate if and only if they have the same norm and the same real part.
\end{proof}

We also need the following elementary result.

\begin{lemma}\label{ACOTACION}
Let $M\in\mathcal{M}(n+1,\R)$ be a real matrix with rows $\sigma_1,\dots,\sigma_{n+1}$ and let $w\in\R^{n+1}$, $w\not=0$. Suppose that $M$ has maximal rank $n+1$ and that its rows have euclidean norm $1$. Then,
$$\vert M\cdot w\vert<\sqrt{n+1}\,\vert w\vert.$$
\end{lemma} 

\begin{proof}
Let the matrix $M=(m_{ij})$ and let $w=\sum\nolimits_{j=1}^{n+1} w_j e_j$ where $e_j$ are the vectors of the canonical basis. Then
\begin{eqnarray}\label{DESIG}
\vert M\cdot w \vert &=&\nonumber\\
  \vert \sum\limits_{i=1}^{n+1} (\sum\limits_{j=1}^{n+1} m_{ij}w_j)e_i\vert 
&=&\nonumber\\
\left[ \sum\limits_{i=1}^{n+1}
(\sum\limits_{j=1}^{n+1} w_j m_{ij})^2\right]^{1/2} 
&=& \\
\left[\sum\limits_{i=1}^{n+1} \langle w,\sigma_i\rangle^2\right]^{1/2}\nonumber,
\end{eqnarray}
where $\langle,\rangle$ is the scalar product in $\R^{n+1}$.
Moreover, for any $1\leq i\leq n+1$ we have
$$
\langle w,\sigma_i\rangle 
=\vert w\vert\, \vert\sigma_i\vert  \cos\sphericalangle(w,\sigma_i)= \vert w \vert\cos\sphericalangle(w,\sigma_i)\leqslant \vert w\vert
$$
and equality implies that   $w$ is a multiple of $\sigma_i$. Since the rows $\sigma_i$ are $\R$-independent by hypothesis, the vector $w\not=0$ can not be in the direction of the $n+1$ rows at the same time, so $\langle w,\sigma_i\rangle <   \vert w\vert$ for some $i$. Then, from  (\ref{DESIG}), 
$$\vert M\cdot w\vert <\left((n+1)\vert w\vert^2\right)^{1/2}=\sqrt{n+1}\vert w \vert.$$
\end{proof}

Next theorem is the main result of this paper.

\begin{theorem}\label{CUATRONO} Let $\sigma_1,\dots,\sigma_4$ be four  quaternions with norm $1$. Then there exists a matrix $A\in Sp(2)$ for which  those  quaternions   are left eigenvalues.  
\end{theorem} 

\begin{proof}
Accordingly to Theorem \ref{INFINITE}
the matrix  must be of the form $A=L_q\circ R_\theta$, $\vert q \vert =1$, $\sin\theta\neq 0$.

Now,  part (3) of Lemma \ref{REAL} means that, in order to find $A$, we have to fix a possible $\cos\theta\neq \pm1$ --the exact value of $\cos\theta$ will be determined later--,
 and then to solve the system of linear equations
\begin{equation}\label{SISTEMA}
\Re(\bar{q}\sigma_m)=\cos\theta,\quad m=1,\dots, 4.
\end{equation}
Moreover, the solution $q$ must verify $\vert q\vert =1$.
 
Let us write
$$q=t+xi+yj+zk, \quad t,x,y,z\in\R,$$ and analogously 
$$\sigma_m=t_m+x_mi+y_mj+z_mk, \quad m=1,\dots, 4.$$ 
Then  system (\ref{SISTEMA}) can  be written as

$$
\pmatrix{
t_1 & x_1& y_1 & z_1\cr
\vdots&\vdots&\vdots&\vdots\cr
t_4&x_4&y_4&z_4\cr
}
\cdot
\pmatrix{
t\cr 
x\cr 
y\cr 
z\cr
}
=
\cos\theta\cdot 
\pmatrix{
1\cr 1\cr 1\cr 1\cr
}
$$
  
which we abbreviate as 
\begin{equation}\label{SISTEMA2}
M\cdot q = \cos \theta\cdot u.
\end{equation}
  
  If  $0<\rank M<4$ we   take $\cos\theta=0$, so system (\ref{SISTEMA2}) is homogeneous. The set of solutions is a non trivial vector space, hence it contains at least one solution $q$ with norm $\vert q \vert=1$. 
 Then a matrix having    $\sigma_1,\dots,\sigma_4$ as eigenvalues would be, for instance, $A=L_q\circ R_{\pi/2}.$
 
 On the other hand, if  the matrix $M$ has maximal rank $4$,  it is invertible and the unique solution of (\ref{SISTEMA2}) is given by 
\begin{equation}\label{SOLUCION}
q=\cos\theta \cdot M^{-1}\cdot u.
\end{equation}
By Lemma \ref{ACOTACION} above we have 
$$2=\vert uÊ\vert = \vert M\cdot M^{-1}\cdot u\vert < \sqrt{4}\,\vert M^{-1}\cdot u \vert,$$
hence $\vert M^{-1}u\vert>1$ and so we can choose $\theta$ such that \begin{equation}\label{ANGLE}
0<\vert\cos\theta\vert=\frac{1}{\vert M^{-1}u\vert}<1
\end{equation}
which implies
$$\vert q\vert=\vert \cos\theta\vert \vert M^{-1}u\vert=1.$$ 
With that angle $\theta$ and the solution $q$ in (\ref{SOLUCION}) we have obtained a matrix $A=L_q\circ R_\theta$ having
$\sigma_m$, $m=1,\dots,4$   among its eigenvalues. 
 \end{proof}

 \begin{example}\label{EJEMPLO}
The four quaternions $1,i,j,k$  give rise to the system
$$\id \cdot q=\cos\theta \cdot u,$$
with unitary solutions $q=\pm (1/2) u$
(we are using (\ref{SOLUCION}) and (\ref{ANGLE})). Then, the only  two symplectic matrices having those four numbers as left eigenvalues are 
$$\frac{1}{4}\left(\begin{array}{cc}u & -\sqrt{3}u\\ 
\sqrt{3}u &u\end{array}\right)$$
and
$$\frac{1}{4}\left(\begin{array}{cc}u & \sqrt{3}u \\ -\sqrt{3}u & u\end{array}\right),$$
where $u=1+i+j+k$.
\end{example}
 
\section{Application to LS category}\label{LS}
%A Theory section should extend, not repeat, the background to the article already dealt with in the Introduction and lay the foundation for further work. In contrast, a Calculation section represents a practical development from a theoretical basis.

The Lusternik-Schnirelmann category of a topological space $X$, denoted by $\cat X$, is the minimum number of open sets (minus one), contractible in $X$, which are needed to cover $X$. This  homotopical invariant  has been widely studied  and has many applications which go from  the calculus of variations to robotics, see \cite{DANIEL,FARBER,JAMES}. The computation of the LS category of Lie groups and homogeneous spaces is a  central problem in this field, where many questions are still unanswered. For instance, for the symplectic groups the only known results are $\cat \SP(1)=1$, $\cat \SP(2)=3$ and $\cat \SP(3)=5$ \cite{TATO}. 

There is a standard technique which has been successfully applied in the complex setting, for instance to the unitary group $U(n)$ \cite{SINGHOF}  and to the symmetric spaces  $U(2n)/Sp(n)$ and $U(n)/O(n)$ \cite{MISU}.  It consists in considering, for a given complex number $z$ with $\vert z \vert=1$, the set $\Omega(z)$ of unitary matrices $A$ such that $A-z I$ is invertible. It turns out that this set is contractible. So, since a unitary $n\times n$ matrix can not  have simultaneously $n+1$ different eigenvalues, it  is possible to cover   the cited spaces  by $n+1$ contractible open sets $\Omega(z_1),\dots,\Omega(z_{n+1})$, showing that they have category $\leq n$ (that $n$ is also a lower bound can be proved with homological methods). 

In the quaternionic setting, we must consider   {\em left} eigenvalues. If $\sigma\in\HH$, the open set
$$\Omega(\sigma)=\{A\in Sp(2)\colon A-\sigma  I \mathrm{\ is \ invertible}\}$$ 
is contractible, for instance by means of the Cayley contraction
$$A_t=\frac{(1+t)A-(1-t)\sigma I}{(1+t)I-(1-t)\bar \sigma A}, \quad t\in [0,1]$$
 (see \cite{MACPERTATO} for a general discussion). Hence our main result  in this paper (Theorem \ref{CUATRONO}) implies that four contractible sets of the type $\Omega(\sigma)$ will never cover $Sp(2)$, despite  the fact that $\cat Sp(2)=3$.  
  
\begin{remark} We observe (cf. Lemma \ref{REAL}) that all the eigenvalues $\sigma$ of the two matrices in Example \ref{EJEMPLO} verify  $\Re(\overline{q}\sigma)=\pm 1/2$.  Then if we take 
$\sigma_5=(i+j)/\sqrt{2}$, we have $\Re (\overline{q}\sigma_5)= \pm 1/\sqrt{2}$, hence the two matrices belong to $\Omega(\sigma_5)$. So five open sets associated to eigenvalues do suffice to cover the group.
\end{remark}

% \appendix

% \section{}
% \label{}

\end{document}